\documentclass[12pt]{amsart}
\usepackage{amsmath,amssymb,amsfonts,amsthm,amsopn}
\usepackage{graphics}
\usepackage{xcolor}
\usepackage{pb-diagram}
\usepackage[title]{appendix}

 \def\Spnr{Sp(d,\R)}
 \def\Gltwonr{GL(2d,\R)}

\newcommand{\tf}{time-frequency}

\newcommand{\tfs}{time-frequency shift}

\newtheorem{theorem}{Theorem}[section]
\newtheorem{lemma}[theorem]{Lemma}
\newtheorem{corollary}[theorem]{Corollary}
\newtheorem{proposition}[theorem]{Proposition}
\newtheorem{definition}[theorem]{Definition}

\newtheorem{remark}[theorem]{Remark}

\newcommand{\beqa}{\begin{eqnarray*}}
\newcommand{\eeqa}{\end{eqnarray*}}

\newcommand{\field}[1]{\mathbb{#1}}
\newcommand{\bR}{\field{R}}        
\newcommand{\bZ}{\field{Z}}        
\newcommand{\bC}{\field{C}}        
\newcommand{\bT}{\field{T}}        %
\def\la{\lambda}

\def\cF{\mathcal{F}}              
\def\cS{\mathcal{S}}
\def\cD{\mathcal{D}}

\def\cB{\mathcal{B}}

\def\cG{\mathcal{G}}
\def\cM{\mathcal{M}}

\def\cA{\mathcal{A}}
\def\cJ{\mathcal{J}}
\def\cI{\mathcal{I}}
\def\cC{\mathcal{C}}

\def\rd{\bR^d}

\def\rdd{{\bR^{2d}}}
\def\zdd{{\bZ^{2d}}}

\def\lrd{L^2(\rd)}

\def\zd{\bZ^d}

\def\spdr{{\mathfrak {sp}}(d,\R)}

\def\R{\right)}

\def\<{\left<}
\def\>{\right>}

\def\inv{^{-1}}

\def\mv1{M_v^1}

\def\phas{(x,\o )}
\def\mn{(m,n)}
\def\mn'{(m',n')}

\def\Spnr{Sp(d,\R)}

\newcommand{\norm}[1]{\lVert#1\rVert}

\def\o{\eta}

\def\R{\mathbb{R}}
\def\Ren{\mathbb{R}^d}

\def\sch{\mathcal{S}}

\def\Sn2{S_{2}(L^{2}(\Ren))}
\def\S1{S_{1}(L^{2}(\Ren))}
\def\sig00{\sigma_{0,0}}

\def\la{\langle}
\def\ra{\rangle}

\newcommand{\A}{\mathcal{A}}

\begin{document}
\begin{abstract} We generalize the results for Banach algebras of pseudodifferential operators obtained by Gr\"ochenig and Rzeszotnik in \cite{GR}  to quasi-algebras of Fourier integral operators. Namely, we introduce quasi-Banach algebras of symbol classes  for Fourier integral operators that we call generalized metaplectic operators, including pseudodifferential operators. This terminology stems from the pioneering work on Wiener algebras of Fourier integral operators \cite{CGNRJMPA}, which we generalize to our framework. This theory finds applications in the study of evolution equations such as  the Cauchy problem for the Schr\"odinger  equation   with bounded perturbations, cf. \cite{CGRPartII2022}.
\end{abstract}

\title[Quasi-Banach algebras of operators]{Quasi-Banach algebras and Wiener properties for pseudodifferential  and generalized metaplectic operators}

\author{Elena Cordero}
\address{Universit\`a di Torino, Dipartimento di Matematica, via Carlo Alberto 10, 10123 Torino, Italy}
\email{elena.cordero@unito.it}
\author{Gianluca Giacchi}
\address{Universit\'a di Bologna, Dipartimento di Matematica, Piazza di Porta San Donato 5, 40126 Bologna, Italy; University of Lausanne, Switzerland; HES-SO School of Engineering, Rue De L'Industrie 21, Sion, Switzerland; Centre Hospitalier Universitaire Vaudois, Switzerland}
\email{gianluca.giacchi2@unibo.it}
\thanks{}

\maketitle

\section{Introduction}

The main characters of this study are the spaces of sequences 
 $$\mathcal{B}=\ell^q_{v_s}(\Lambda),  \quad 0<q<1, \,s\in\bR,$$
 for $\Lambda=A\zdd$, $A\in GL(d,\bR)$, a given lattice.  Namely, $a=(a_{\lambda})_{\lambda\in\Lambda}\in \mathcal{B} $ if the quasi-norm 
 $$\|a\|_{\ell^{q}_{v_s}}=\left(\sum_{\lambda\in\Lambda}|a_{\lambda}|^q v_s(\lambda)^q\right)^{\frac1 q}
 $$
 is finite,  with $v_s(\lambda)=(1+|\lambda|)^s$. 
The spaces $(\ell^q_{v_s}(\Lambda), \|\cdot\|_{v_s})$ are quasi-Banach spaces, with quasi-norms satisfying $$\|a+b\|_{\ell^q_{v_s}}^q\leq \|a\|_{\ell^q_{v_s}}^q+ \|b\|_{\ell^q_{v_s}}^q,\quad a,b\in \ell^q_{v_s}(\Lambda).$$ 
 They also  enjoy the algebra property (w.r.t. the discrete convolution): 
 $$ \|a\ast b\|_{\ell^q_{v_s}} \leq \|a\|_{\ell^q_{v_s}} \|b\|_{\ell^q_{v_s}},\quad a,b\in \ell^q_{v_s}(\Lambda).$$
 The sequence $\delta=(\delta(\lambda))_{\lambda\in\Lambda}$, given by $\delta(\lambda)=1$ for $\lambda=0$ and $\delta(\lambda)=0$ for $\lambda\in\Lambda\setminus\{0\}$, is the unit element.
 
 These spaces of sequences have manifold applications. For instance, they play a crucial role in the sparsity estimates for PDE's (see, e.g., \cite{CD2005,CD2007,sparsity2009} and references therein)  and are widely employed in approximation theory \cite{DeVore}.

 In our framework, they are the key tool to define classes of operators which behave \emph{nicely} since are subclasses  of bounded operators on $\lrd$, enjoying many valuable properties. 
The pioneering work in this direction is due to Gr\"ochenig \cite{charly06}, where he proved the algebra property and  the inverse-closedness of   pseudodifferential operators having symbols in the  Sj\"ostrand Class \cite{A76,Sjo95}. These results were further extended to other algebras of operators by Gr\"ochenig and Rzeszotnik in \cite{GR}. The latter work is our main source of inspiration.

The function spaces of our study are the Wiener amalgam spaces $W(C,\ell^q_{v_s})(\rdd)$ $0<q\leq1$, defined in terms of the (quasi-)norms
\[
\norm{F}_{W(C,\ell^q_{v_s})}=\left(\sum_{k\in\zdd}(\sup_{z\in [0,1]^{2d}}|F(z+k)|)^qv_s(k)^q\right)^{1/q}
\]
as the spaces of continuous functions $F$ on $\rdd$ such that $\norm{F}_{W(C,\ell^q_{v_s})}<\infty$ (here $\Lambda=\zdd$, see the next section for a general lattice).

Fix $\chi$ a matrix in the symplectic group $Sp(d,\bR)$ (see the next section for its definition) and $0<q\leq1$; we say that $T:\cS(\rd)\to\cS'(\rd)$ is in the class $FIO(\chi,q,v_s)$ if there exists $H\in W(C,\ell^q_{v_s})(\rdd)$ and $g\in \cS(\rd)\setminus\{0\}$ such that
\begin{equation}\label{decaycond}
	|\langle T\pi(z)g,\pi(w)g\rangle|\leq H(w-\chi z),\quad w,z\in\rdd,
\end{equation}
where $\pi(z)g(t):=e^{2\pi i\xi t}g(t-x)$ ($z=(x,\xi)\in\rdd$) are the time-frequency shifts of $g$.

 An operator $T$ that lies in $FIO(\chi,q,v_s)$  is called \textit{generalized metaplectic operator}. 
 
 The case $q=1$ corresponds to the class of generalized metaplectic operators introduced in \cite{CGNRJMPA} which gives rise to Banach algebras of operators bounded on $\lrd$. These algebras can be successfully applied to the study of Schr\"odinger  equations \cite{CGNRJMP2014}.
 
 We extend the results above to the case $0<q<1$, that is the quasi-algebra case. Similarly to algebras, these operators can be applied to study Schr\"odinger  equations with bounded perturbation, cf. \cite{CGRPartII2022}.

 

Our work concerns the study of the main properties of $FIO(\chi,q,v_s)$. This requires a lot of technicalities. As far as we know, the theory of quasi-Banach algebras developed so far is very poor. So the main work here is to infer all the properties for quasi-Banach algebras we need for our class of operators.

In the first part of the paper we focus on the quasi-algebras $\mathcal{B}=\ell^q_{v_s}(\Lambda)$. We carry the definition of the matrix quasi-algebras $\cC_\cB$ given in \cite{GR} to $\mathcal{B}$ as follows:
\[
	A=(a_{\lambda,\mu})_{\lambda,\mu\in\Lambda}\in\cC_\cB \qquad \Longleftrightarrow \qquad (\sup_{\lambda\in\Lambda}|a_{\lambda,\lambda-\mu}|)_{\mu\in\Lambda}\in\cB.
\]
We prove that \cite[Theorem 3.2]{GR} generalizes to the quasi-algebras setting:
\begin{theorem}\label{I1} The following are equivalent:\\
	$(i)$ $\cB$ is inverse-closed in $B(\ell^2)$;\\
	$(ii)$ $\cC_\cB$ is inverse-closed in $B(\ell^2)$;\\
	$(iii)$ The spectrum $\widehat\cB\simeq\bT^d$.
\end{theorem}

Then, we turn to the almost diagonalization of Weyl operators. Briefly, the time-frequency representation defined for all $f,g\in\cS(\rd)$ as
\[
	W(f,g)(x,\xi)=\int_{\rd}f(x+\frac{t}{2})\overline{g(x-\frac{t}{2})}e^{-2\pi i\xi\cdot t}dt,
\]
is called (cross-) Wigner distribution, and was first introduced by Wigner in 1932 in Quantum Mechanics, cf. \cite{Wigner}.

Since $W:\cS(\rd)\times\cS(\rd)\to\cS(\rdd)$, for $\sigma\in\cS'(\rdd)$ we can define the Weyl pseudodifferential operator $Op_w(\sigma):\cS(\rd)\to\cS'(\rd)$ with symbol $\sigma$ by
\[
	\langle Op_w(\sigma)f,g\rangle=\langle \sigma,W(g,f)\rangle \qquad f,g\in\cS(\rd).
\]
 We are interested in characterizing the invertibility properties of  $Op_w(\sigma)$ in terms of its Gabor matrix $M(\sigma)$:
\[
	M(\sigma)_{\mu,\lambda}=\langle Op_w(\sigma)\pi(\lambda)g,\pi(\mu)g\rangle, \qquad \lambda,\mu\in\Lambda
\]
(see \eqref{unobis2s} below for its definition). In general, interesting conclusions follow when conditions on $\sigma$ are imposed, such as their membership to some distributional space. In time-frequency analysis, modulation spaces are used to measure the time-frequency content of tempered distributions. They were introduced by Feichtinger in 1983, cf. \cite{F1}, and later extended to the quasi-Banach setting by Galperin and Samarah, cf. \cite{Galperin2004}. Namely, if $0<p,q\leq\infty$, $g\in\cS(\rd)$ and $m$ is a $v_s$-moderate weight function, a tempered distribution $f\in \cS'(\rd)$ belongs to $M^{p,q}_m(\rd)$ if
\[
	\norm{f}_{M^{p,q}_m}=\norm{V_gf}_{L^{p,q}_m}<\infty,
\]
where $V_gf$ is the Short-Time Fourier transform of $f$ with respect to the window $g$, i.e.
\[
	V_gf(x,\xi)=\langle f,\pi(x,\xi)g\rangle,\quad (x,\xi)\in\rdd.
\]
We prove the inverse-closedness in $B(L^2(\rd))$ of the class of Weyl operators with symbols  $\sigma\in M^{\infty,q}_{1\otimes v_s}(\rdd)$ (see Theorem \ref{6.8} below):
\begin{theorem}\label{I2}
	If $\sigma\in M^{\infty,q}_{1\otimes v_s}(\rdd)$, $0<q\leq1$ and $Op_w(\sigma)$ is invertible on $L^2(\rd)$, then $(Op_w(\sigma))^{-1}=Op_w(b)$ for some $b\in M^{\infty,q}_{1\otimes v_s}(\rdd)$.
\end{theorem}

The theory developed so far finds application to  generalized metaplectic operators. Namely, we first prove the invertibility property in the class $FIO(\chi,q,s)$ (cf. Theorem \ref{inverse}):
\begin{theorem}
	Consider $T\in FIO(\chi,q,v_s),$ such that $T$ is invertible on $L^2(\rd)$,
	then $T^{-1} \in FIO(\chi^{-1},q,v_s)$.
\end{theorem}
In other words, the class of generalized metaplectic operators is closed under inversion.

Further, observe that the decay condition (\ref{decaycond}) alone does not provide an explicit expression for a generalized metaplectic operator. We prove that if $T\in FIO(\chi,q,v_s)$, $0<q\leq 1$, then
\begin{equation}\label{I3}
	T=Op_w(\sigma_1)\mu(\chi) \quad and \quad T=\mu(\chi)Op_w(\sigma_2)
\end{equation}
for $\sigma_1\in M^{\infty,q}_{1\otimes v_s}(\rdd)$, $\sigma_2=\sigma_1\circ\chi$ and $\mu(\chi)$ the metaplectic operator associated to the symplectic matrix $\chi$ (cf. Theorem \ref{pseudomu} below). This provides an explicit expression for operators in  $FIO(\chi,q,v_s)$. \\

This work is divided as follows: notation and preliminaries are established in Section \ref{sec:preliminaries}, where we also justify the importance of the quasi-Banach setting. Section \ref{sec:QAGMO} is devoted to the definition of generalized metaplectic operators, their extensions to bounded operators on modulation spaces $M^p_m(\rd)$, and the proof of Theorem \ref{I1}. In Section \ref{sec:ADWO} we study the matrix operators associated to Weyl pseudodifferential operators with symbols in $M^{\infty,q}_{1\otimes v_s}(\rdd)$ and prove Theorem \ref{I2}. In Section \ref{sec:GMO} we prove both that the class of generalized metaplectic operators $FIO(Sp(d,\bR),q,v_s)$ is closed under inversion and (\ref{I3}). To prove these results, we need to extend the theory of Banach-algebras to the quasi-Banach algebras setting. We carefully check the main issues and detail the differences in the Appendix.

\section{Preliminaries}\label{sec:preliminaries}
\textbf{Notation.} We denote $t^2=t\cdot t$,  $t\in\rd$, and
$xy=x\cdot y$ (scalar product on $\Ren$).  The space   $\sch(\Ren)$ is the Schwartz class whereas $\sch'(\Ren)$  the space of temperate distributions.   The brackets  $\la f,g\ra$ denote the extension to $\sch' (\Ren)\times\sch (\Ren)$ of the inner product $\la f,g\ra=\int f(t){\overline {g(t)}}dt$ on $L^2(\Ren)$ (conjugate-linear in the second component).We write a point in the phase space (in the \tf\ space) as
$z=(x,\eta)\in\rdd$, and  the corresponding phase-space shift (\tfs )
acts on a function or distribution  as
\begin{equation}
\label{eq:kh25}
\pi (z)f(t) = e^{2\pi i \eta t} f(t-x) \, .
\end{equation}
We shall work with  lattices in the phase-space $\Lambda\subset \rdd$,  $\Lambda=A\zdd$,  with $A\in GL(2d,\R)$ and we will denote by $Q$ a fundamental domain containing the origin.  $\cC_0^\infty(\rdd)$ denotes the space of smooth functions with compact support.
\subsection{The symplectic group $Sp(d,\mathbb{R})$ and the metaplectic operators}
We recall definitions and properties of symplectic matrices and metaplectic operators  in a nutshell, referring to \cite{Gos11} for  details. First, we write  $\Gltwonr$ for the group of $2d\times 2d$ real invertible matrices.
The standard symplectic matrix is denoted by
\begin{equation}\label{J}
J=\begin{pmatrix} 0_{d\times d}&I_{d\times d}\\-I_{d\times d}&0_{d\times d}\end{pmatrix},
\end{equation}
The symplectic group is
\begin{equation}\label{defsymplectic}
\Spnr=\left\{\cA\in\Gltwonr:\;\cA^T J\cA=J\right\},
\end{equation}
where  $\cA^T$ is the transpose of $\cA$.\par
The symplectic  algebra 
$\spdr$ is the set of
$2d\times 2d$ real matrices
$\A$   such that $e^{t \A}
\in \Spnr$, for every $t\in\R$.\par
The metaplectic representation
$\mu$ is a unitary representation of  (the double cover of) $\Spnr$ on $\lrd$.
For elements of $\Spnr$ of special form the metaplectic
representation can be computed explicitly. That is to say, for
$f\in L^2(\R^d)$,  $C$ real symmetric $d\times d$ matrix ($C^T=C$) we consider the symplectic matrix 
\begin{equation}\label{Vc}
V_C=\begin{pmatrix}
I_{d\times d} & 0_{d\times d}\\
C & I_{d\times d} 
\end{pmatrix};
\end{equation}
then, up to a phase factor,
\begin{equation}\label{muVc}
\mu(V_C)f(t)=e^{i\pi Ct\cdot t}f(t)
\end{equation}
for all $f\in L^2(\rd)$. 
For the standard matrix $J$ in \eqref{J},
\begin{equation}
\mu(J)f=\cF f\label{muJ};
\end{equation}
Fix $L\in GL(d,\bR)$ and consider the related the symplectic matrix 
\begin{equation}\label{MotL}
\cD_L= \left(\begin{array}{cc}
L^{-1} &0_{d\times d}\\
0_{d\times d} & L^T
\end{array}\right)  \in Sp(d,\bR);
\end{equation}
 up to a phase factor we have
\begin{equation}\label{AdL}
\mu(\cD_L)F(t)=\sqrt{|\det L|}F(Lt)=\mathfrak{T}_{L} F(t),\quad F\in\lrd.
\end{equation}
The metaplectic operators have a group structure with respect to the composition.
\begin{proposition}\label{deGosson96}
	The metaplectic group is generated by the operators $\mu(J), \mu(\cD_L)$ and $\mu(V_C)$.
\end{proposition}

The relation between \tfs s and metaplectic operators is the following:

\begin{equation}\label{metap}
\pi (\cA z) = c_\cA  \, \mu (\cA ) \pi (z) \mu (\cA )\inv  \quad  \forall
z\in \rdd \, ,
\end{equation}
with a phase factor $c_\cA \in \bC , |c_{\cA }| =1$  (for details, see e.g. \cite{folland89,Gos11}).

\subsection{Function Spaces}
We shall work with  lattices $\Lambda =A\zdd$, with $A\in GL(2d,\bR)$ and define the spaces of sequences accordingly.

We denote by $v$  a continuous, positive,  submultiplicative  weight function on $\rdd$, i.e., 
$ v(z_1+z_2)\leq v(z_1)v(z_2)$, for all $ z_1,z_2\in\rdd$.
We say that $w\in \mathcal{M}_v(\rdd)$ if $w$ is a positive, continuous  weight function  on $\rdd$  {\it
	$v$-moderate}:
$ w(z_1+z_2)\leq Cv(z_1)w(z_2)$  for all $z_1,z_2\in\rdd$.

We denote by $\mathcal{M}_v(\Lambda)$ the restriction of weights $w\in \mathcal{M}_v(\rd)$ to the lattice $\Lambda$. 
We will mainly work with polynomial weights of the type
\begin{equation}\label{vs}
	v_s(z)=\la z\ra^s =(1+|z|)^{s}\quad (v_s(\lambda)=(1+|\lambda|)^{s}),\quad s\in\bR,\quad z\in\rdd\,\, (\lambda\in \Lambda).
\end{equation}

We define $(w_1\otimes w_2)\phas= w_1(x)w_2(\o)$, for $w_1,w_2$ weights on $\rd$.\par

\begin{definition}\label{ellp}
	For $0<q\leq \infty$, $m\in \mathcal{M}_v(\Lambda)$, the space $\ell^{q}_m(\Lambda)$ consists of all sequences $a=(a_{\lambda})_{\lambda\in\Lambda}$ for which the (quasi-)norm 
	$$\|a\|_{\ell^{q}_m}=\left(\sum_{\lambda\in\Lambda}|a_{\lambda}|^q m(\lambda)^q\right)^{\frac1 q}
	$$
is finite	(with obvious modification for $q=\infty$).
\end{definition}

Here there are some properties we need in the sequel \cite{Galperin2014,Galperin2004}:
\begin{itemize}
	\item[(i)] \emph{Inclusion relations}: If $0<q_1\leq q_2 \leq\infty$, then $\ell^{q_1}_m(\Lambda)\hookrightarrow\ell^{q_2}_m(\Lambda)$, for any positive weight function $m$ on $\Lambda$.
	\item[(ii)] \emph{Young's convolution inequality}: Consider $m\in\mathcal{M}_v(\Lambda)$, $0<p,q,r\leq \infty$ with
	\begin{equation}\label{Yrgrande1}
	\frac1p+\frac1q=1+\frac1r, \quad \mbox{for}\quad 1\leq r\leq \infty
	\end{equation}
	and
	\begin{equation}\label{Yrminor1}
	p=q=r, \quad \mbox{for}\quad 0<r<1.
	\end{equation}
	Then for all $a\in \ell^p_m(\Lambda)$ and $b\in\ell^q_v(\Lambda)$, we have $a\ast b\in \ell^r_m(\Lambda)$, with
	\begin{equation*}
	\|a\ast b\|_{\ell^r_m}\leq C \|a\|_{\ell^p_m}\|b\|_{\ell^q_v},
	\end{equation*}
	where $C$ is independent of $p,q,r$, $a$ and $b$. If $m\equiv v\equiv1$, then $C=1$. 
	\item[(iii)] \emph{H\"{o}lder's inequality}: For any positive weight function $m$ on $\Lambda$, $0<p,q,r\leq \infty$, with $1/p+1/q=1/r$, 
	\begin{equation}\label{ptwellp}
	\ell^p_m (\Lambda)\cdot \ell^q_{1/m}(\Lambda)\hookrightarrow \ell^r(\Lambda),
	\end{equation}
	where the symbol $\hookrightarrow$ denotes that the inclusion is a  continuous mapping.
\end{itemize}
\subsection{Wiener Amalgam Spaces \cite{Feichtinger_1981_Banach,feichtinger-wiener-type,Feichtinger_1990_Generalized,Galperin2004,Rauhut2007Winer}.}
Let $B$ one of the
following Banach spaces:  $C(\rd)$ (space of continuous functions on $\rdd$),
$L^p(\rdd)$, $1\leq  p\leq \infty$; 
let $C$ be one of the following (quasi-)Banach spaces: $\ell^q_m(\Lambda)$, $0<
q\leq\infty$, $m\in\cM_v(\Lambda)$.

For any given function $f$ which is locally in $B$ (i.e. $g f\in B$,
$\forall g\in\cC_0^\infty(\rdd)$), we set $f_B(x)=\| fT_x g\|_B$.
The {\it Wiener amalgam space} $W(B,C)$ with local component $B$ and
global component  $C$ is defined as the space of all functions $f$
locally in $B$ such that $f_B\in C$. Endowed with the (quasi-)norm
$\|f\|_{W(B,C)}=\|f_B\|_C$, $W(B,C)$ is a (quasi-)Banach space. Moreover,
different choices of $g\in \cC_0^\infty(\rdd)$  generate the same space
and yield equivalent norms.

In particular, for $s\geq 0$ and $\Lambda=A\zdd$, with $Q$ fundamental domain containing the origin,  $A\in GL(2d,\bR)$, $\ell^q_{v_s}=\ell^q_{v_s}(\Lambda)$, we shall consider the Wiener amalgam space $W(C,\ell^q_{v_s})(\rdd)$, the space of continuous functions $F$ on $\rdd$ such that
\begin{equation}\label{Wiener-space}
\|F\|_{{W(C,\ell^q_{v_s})}}=\left(\sum_{\lambda\in\Lambda}(\sup_{z\in Q}|F(z+\lambda)|)^qv_s(\lambda)^q\right)^{\frac1q}<\infty
\end{equation}
(evident changes for $q=\infty$),
where $v_s$ is defined in \eqref{vs}. Let us recall that  $v_s$ is submultiplicative for $s\geq0$.

\begin{lemma}\label{WA}
	Let $B_i$, $C_i$, $i\in\{1,2,3\}$, be (quasi-)Banach spaces  as defined  above.
	\begin{itemize}
		\item[(i)] \emph{Convolution.}
		If $B_1\ast B_2\hookrightarrow B_3$ and $C_1\ast
		C_2\hookrightarrow C_3$, then
		\begin{equation}\label{conv0}
		W(B_1,C_1)\ast W(B_2,C_2)\hookrightarrow W(B_3,C_3).
		\end{equation}
		\item[(ii)]\emph{Inclusions.} If $B_1 \hookrightarrow B_2$ and $C_1 \hookrightarrow C_2$ then
		\begin{equation*}
		W(B_1,C_1)\hookrightarrow W(B_2,C_2).
		\end{equation*}
		\noindent Moreover, the inclusion of $B_1$ into $B_2$ need only hold ``locally'' and the inclusion of $C_1 $ into $C_2$  ``globally''.
		Specifically for $\ell^q_{v_s}$, $s\geq0$,  if we take $0<q_i\leq\infty$, $i=1,2$, then
		\begin{equation}\label{lp}
		q_1\geq q_2\,\Longrightarrow W(C,\ell^{q_1}_{v_s})\hookrightarrow
		W(C,\ell^{q_2}_{v_s}).
		\end{equation}
\end{itemize}
\end{lemma}
For the quasi-algebras of FIOs we shall use the following lemma.
\begin{lemma}\label{convWienerlq}
	For $0< q\leq 1$, $s\geq0$, we have
	\begin{equation}\label{convWienerlqeq}
	W(C,\ell^q_{v_s})\ast 	W(C,\ell^q_{v_s})\hookrightarrow 	W(C,\ell^q_{v_s}).
	\end{equation}
\end{lemma}
\begin{proof}
	It follows from the convolution and the inclusion relations in Lemma \ref{WA}.  Namely,
	$$W(C,\ell^q_{v_s})\hookrightarrow W(L^1,\ell^q_{v_s})$$
	since $C(\rdd)\hookrightarrow L^1(\rdd)$ locally. Hence, the convolution relations give
	$$ W(C,\ell^q_{v_s})\ast W(C,\ell^q_{v_s}) \hookrightarrow W(C,\ell^p_{v_s})\ast W(L^1,\ell^q_{v_s}) \hookrightarrow  W(C,\ell^q_{v_s}),$$
	since 
	$C(\rdd)\ast L^1(\rdd)\hookrightarrow C(\rdd)$ and $\ell^q_{v_s}\ast \ell^q_{v_s}{\hookrightarrow}\ell^q_{v_s}$, $s\geq0$, $0<q\leq1$, by the Young's convolution inequalities.
\end{proof}

\begin{lemma}\label{lemmaRicoprimento}
Let $s\in\mathbb{R}$, $0<q\leq\infty$ and $M\in GL(2d,\mathbb{R})$. Then, $W(C,\ell^q_{v_s})(\rdd)$ is invariant under $\chi$, i.e. if $H\in W(C,\ell^q_{v_s})(\rdd)$, then $H\circ M\in W(C,\ell^q_{v_s})(\rdd)$.
\end{lemma}
\begin{proof}
	Clearly, $H\circ M$ is continuous. Assuming $q\neq\infty$, 
	\[\begin{split}
		\norm{H\circ M}_{W(C,\ell^q)}^q&=\sum_{\lambda\in\Lambda}(\sup_{z\in Q}|H(M(z+\lambda))|)^qv_s(\lambda)^q\\
		&=\sum_{\lambda\in\Lambda}(\sup_{z\in E_\lambda}|H(z)|)^qv_s(\lambda)^q,
	\end{split}\]
	where $E_\lambda=M(Q+\Lambda)$. For all $\lambda\in\Lambda$, let $\mathcal{R}_\lambda:=\{Q_{\mu}^{(\lambda)}\}_{\mu\in\Lambda}$ be the smallest finite covering of $E_\lambda$ with  $Q_{\mu}^{(\lambda)}$ of the family $\mathcal{Q}=\{Q+\lambda \ : \ \lambda\in\Lambda\}$. Clearly, $\beta:=\sup_{\lambda}\text{card}(\mathcal{R}_\lambda)<\infty$. Then,
	\[
		\begin{split}
		\norm{H\circ \chi}_{W(C,\ell^q_{v_s})}^q&=\sum_{\lambda\in \Lambda}(\sup_{z\in E_\lambda}|H(z)|)^qv_s(\lambda)^q\leq\sum_{\lambda\in\Lambda}\sup_{z\in\bigcup_\mu Q_\mu^{(\lambda)}}|H(z)|^qv_s(\lambda)^q\\
		&\leq \sum_{\lambda\in\Lambda}\sum_{\mu\in \mathcal{R}_\lambda}\sup_{Q_\mu^{(\lambda)}}|H(z)|^qv_s(\lambda)^q\leq \beta\sum_{\lambda\in\Lambda}\sup_{z\in Q^{(\lambda)}}|H(z)|^qv_s(\lambda),
	\end{split}
	\]	
	where $Q^{(\lambda)}$ is any of the subsets $Q_\mu^{(\lambda)}$ that contain $\arg\max_{z\in\bigcup_\mu Q_\mu^{(\lambda)}}|H(z)|^q$. Observe that these points exist because the  sets $Q+\lambda$ are compact since $Q$ is,  and $H$ is continuous. It may happen that $Q^{(\lambda)}=Q^{(h)}$ for $h\neq \lambda$, but clearly a subset $Q^{(\lambda)}$ can belong to at most $2^{2d}|\det A|$ families $\mathcal{R}_\lambda$ (recall $\Lambda =A\zdd$). Therefore, 
	\[\begin{split}
		\norm{H\circ M}_{W(C,\ell^q_{v_s})}^q&\leq \beta\sum_{k\in\Lambda}\sup_{z\in Q^{(\lambda)}}|H(z)|^qv_s(\lambda)\\
		&\leq 4^d|\det A|\beta \sum_{\lambda\in\Lambda}\sup_{z\in \lambda+Q}|H(z)|^q v_s(\lambda)^q\\
		&=4^d |\det A|\beta \norm{H}_{W(C,\ell^q_{v_s})}^q.
	\end{split}\]
	The case $q=\infty$ is trivial.
\end{proof}

\section{Quasi-algebras of generalized metaplectic operators}\label{sec:QAGMO}
This section contains the most interesting results of this manuscript. In fact, we prove in detail the issues used to study Schr\"{o}dinger equations with bounded perturbations \cite{CGRPartII2022}.

We first recall the definition of the (quasi-)algebras of FIOs used there, that extend the algebra definition in the pioneering papers \cite{CGNRJMPA,CGNRJMP2014}. 

Basic tool is the theory of Gabor frames. Consider a lattice  $\Lambda=A\zdd$,  with $A\in GL(2d,\R)$, and a non-zero window function $g\in L^2(\rd)$, then  a \emph{Gabor system} is the sequence: $$\cG(g,\Lambda)=\{\pi(\lambda)g:\
\lambda\in\Lambda\}.$$
A Gabor system $\cG(g,\Lambda)$   is
a Gabor frame if there exist
constants $A, B>0$ such that
\begin{equation}\label{gaborframe}
A\|f\|_2^2\leq\sum_{\lambda\in\Lambda}|\langle f,\pi(\lambda)g\rangle|^2\leq B\|f\|^2_2,\quad \forall f\in L^2(\rd).
\end{equation}
For a Gabor frame $\cG(g,\Lambda)$,  the \emph{Gabor matrix} of a linear continuous operator $T:\cS(\rd)\to \cS'(\rd)$ is defined to be 
\begin{equation}\label{unobis2s} \langle T \pi(z)
g,\pi(u)g\rangle,\quad z,u\in \rdd.
\end{equation}

%


\begin{definition}\label{def1.1} For $\chi \in \Spnr $,
	$g\in\cS(\rd)$, $0< q\leq 1$, a
	linear operator $T:\cS(\rd)\to\cS'(\rd)$ is in the
	class $FIO(\chi,q,v_s)$ if there exists a function $H\in W(C,\ell^q_{v_s})(\rdd)$, 
	such that 
	\begin{equation}\label{asterisco}
	|\langle T \pi(z) g,\pi(w)g\rangle|\leq H(w-\chi z),\qquad \forall w,z\in\rdd.
	\end{equation}
\end{definition}
The union
\[
FIO(Sp(d,\R),q,v_s)=\bigcup_{\chi\in Sp(d,\R)} FIO(\chi,q,v_s)
\]
is named the class of {\bf generalized metaplectic operators}.\par
Arguing similarly to  \cite[Proposition 3.1]{CGNRJMP2014} we  show that the previous definition does not depend on the function $g$. 
\begin{proposition}\label{prop3.1}
	The definition of the class $FIO(\chi,q,v_s)$ is independent of the
	window function $g\in\cS(\rd)$.
\end{proposition}
\begin{proof}
	Assume that (\ref{asterisco}) holds for some window function $g\in\mathcal{S}(\rd)$. We must show that if $\varphi\in\mathcal{S}(\rd)$ is another window function, then we can write 
	\[
		|\langle T\pi(z)\varphi,\pi(w)\varphi\rangle|\leq \tilde H(w-\chi z)
	\]
	for some $\tilde H\in W(C,\ell^q_{v_s})(\rdd)$. The calculation in \cite[Proposition 3.1]{CGNRJMP2014} shows that
	\[
		|\langle T\pi(z)\varphi,\pi(w)\varphi\rangle|\leq\frac{1}{\Vert{g}\Vert_2^4}\int_{\rdd}(H\ast |V_{\mu(\chi)g}\mu(\chi)\varphi|)(r-\chi z)|V_\varphi g(w-r)|dr.
	\]
	By Lemma \ref{convWienerlq}, 
	\[
		G:=H\ast |V_{\mu(\chi)g}\mu(\chi)\varphi|\in W(C,\ell^q_{v_s})(\rdd)\ast \mathcal{S}(\rdd)\subset W(C,\ell^q_{v_s})
	\]
	for all $s\geq0$. Therefore,
\begin{align*}
		|\langle T\pi(z)\varphi,\pi(w)\varphi\rangle|&\leq\frac{1}{\Vert{g}\Vert_2^4}\int_{\rdd} G(r-\chi z)|V_\varphi g(w-r)|dr=G\ast |V_\varphi g|(w-\chi z)\\
		&=:\tilde H(w-\chi z).
	\end{align*}
	Again, by Lemma \ref{convWienerlq}, $\tilde H\in W(C,\ell^q_{v_s})(\rdd)$.
\end{proof}

Let us recall that in the case  $q=1$ the original definition of $FIO(\chi,v_s)$ in \cite{CGNRJMP2014} was formulated for a function $H\in L^1_{v_s}(\rdd)$, instead of the more restrictive condition $H\in W(C,\ell^1_{v_s})(\rdd)$. However, Proposition 3.1 in \cite{CGNRJMP2014} shows that the two definitions are equivalent.

Of interest for applications, is the possibility to rewrite the estimate \eqref{asterisco} in the discrete setting, as explained in the following result. The proof is an easy modification   of the one in  \cite[Theorem 3.1]{CGNRJMPA}, so it is omitted.
\begin{theorem}\label{cara}
Let $\mathcal{G}(g,\Lambda)$ be a Gabor frame with
$g\in\cS(\rd)$.	Consider  a continuous linear operator $T:\cS(\rd)\to\cS'(\rd)$, a matrix
	$\chi\in Sp(d,\bR)$, and parameters $0< q\leq 1$, $s\geq 0$. Then the following conditions are
	equivalent:\par
	{(i)} There exists $H\in W(C,\ell^q_{v_s})(\rdd)$, 
	such that  $T$ satisfies \eqref{asterisco};
	\par {(ii)} There exists  $h\in \ell^q_{v_s}(\Lambda)$, 
	such that 
	\begin{equation}\label{unobis2}
	|\langle T \pi(\lambda) g,\pi(\mu)g\rangle|\leq h( \mu-\chi(\lambda)),\qquad \forall \lambda,\mu\in \Lambda.
	\end{equation}
\end{theorem}
Following the guidelines of the works \cite{CGNRJMPA,CGNRJMP2014} we can exhibit the results below.
\begin{theorem}\label{listprop}
	(i) \textit{\bf Boundedness}. Consider  $\chi\in\Spnr$, $0< q\leq 1$, $s\geq 0$, $m\in \cM_{v_s}(\rdd)$. Let $T$ be
	a generalized metaplectic operator in $FIO(\chi,q,v_s)$. Then $T$ is
	bounded from $M^p_m(\rd)$ to $M^p_{m\circ\chi^{-1}}(\rd)$, for $q\leq p\leq
	\infty$. \\
	(ii) \textit{\bf Algebra property}. Let $\chi_i\in\Spnr$, $s\geq0$ and $T_i\in FIO(\chi_i,q, v_s)$, $i=1,2$. Then $T_1T_2\in FIO(\chi_1\chi_2,q,v_s)$.
\end{theorem}
\begin{proof} $(i)$ Fix $q\leq p\leq\infty$ and a window  $g\in\mathcal{S}(\rd)$ such that $\mathcal{G}(g,\Lambda)$ is a Parseval Gabor frame for $\lrd$.  Using  $T=V_g^\ast V_g T V_g^\ast V_g$, the equivalent discrete (quasi-)norm for the modulation space, see e.g. \cite[Proposition 1.5]{ToftquasiBanach2017}, the estimate in \eqref{unobis2} and Young's convolution inequality $\ell^q_{v_s}\ast \ell^{p}_{m}\hookrightarrow \ell^{p}_{m}$, for $q\leq p$, $0<q\leq 1$,  $m\in \cM_{v_s}(\rdd)$,
\begin{align*}
	\|Tf\|_{M^{p}_{m\circ\chi^{-1}}}&\asymp \|V_g (Tf)\|_{\ell^{p}_{m\circ\chi^{-1}}(\Lambda)}
\leq \||h\circ\chi|\ast |V_g f|(\chi^{-1}(\cdot))\|_{\ell^{p}_{m\circ\chi^{-1}}(\Lambda)}\\
	&\lesssim  \|h\|_{\ell^q_{v_s}(\Lambda)}\|V_g fm\|_{\ell^{p}_{m}(\Lambda)}
	\leq  C \| f\|_{M^{p}_{m}},
\end{align*}	
since  $h\circ \chi\in \ell^q_{v_s}(\Lambda)$.
\par 
$(ii)$ We write $T_1T_2=V_g^\ast(V_g T_1 V_g^\ast)(V_g T_2V_g^\ast)V_g$ and denote with $H_i$ the function controlling the kernel of $T_i$ defined in (\ref{asterisco}) ($i=1,2$). Then, the same computation in \cite[Theorem 3.4]{CGNRJMP2014} gives 
\[	
	|\langle T_1T_2\pi(z)g,\pi(w)g\rangle|\leq (((H_1\circ\chi_1)\ast H_2) \circ \chi_1^{-1})(w-\chi_1\chi_2z),\quad z,w\in\rdd.
\]
 The assertion follows applying Lemmas \ref{convWienerlq} and \ref{lemmaRicoprimento}.
\end{proof}

We next focus on the invertibility property. We use the notations already introduced in \cite{CGRPartII2022}. Let us underline that the algebra cases corresponding to $\ell^1_{v_s}(\Lambda)$ where already treated in \cite{charly06} and \cite{GR} (and references therein).

\begin{definition} [Definition 6.5 \cite{CGRPartII2022}] We define $\cB:=\ell^q_{v_s}(\Lambda)$, $0< q\leq 1$, $s\geq0$. Let $A$ be a matrix on $\Lambda$ with entries $a_{\lambda,\mu}$,  $\lambda,\mu\in \Lambda$, and  $d_A$ be the sequence with entries $d_A(\mu)$ defined by
	\begin{equation}\label{3.2}
	d_A(\mu)=\sup_{\lambda\in\Lambda}|a_{\lambda,\lambda-\mu}|.
	\end{equation} 
	We state that  $A\in\cC_{\cB}$ if $d_A\in\cB$. The (quasi-)norm in $\cC_A$ is given by
	$$\|A\|_{\cC_\cB}=\|d\|_{\cB}.$$
\end{definition}
The value $ d_A(\mu)$ is the supremum of the entries in the $\mu-th$ diagonal of $A$, thus the $\cC_\cB$-norm describes the off-diagonal decay of $A$. We identify an element $b\in\cB\subset\ell^1(\Lambda)$ with the corresponding convolution operator $C_b a=a\ast b$. This allows to treat $\cB$ as a quasi-Banach subalgebra of $\cB(\ell^2(\Lambda))$,  the algebra of bounded operators on $\ell^2(\Lambda)$.\par
The elementary properties of $\cC_{\cB}$ proved for the algebra case $\cB=\ell^1_{v_s}(\Lambda)$ are valid also for the quasi-algebra case  $\cB=\ell^q_{v_s}(\Lambda)$, $0<q<1$. We list them and for their proof we refer to the arguments in Lemma 3.4 in \cite{GR}. 
\begin{lemma}\label{3.4}
For $0<q< 1$ we have that  	$\cB=\ell^q_{v_s}(\Lambda)$ is a solid quasi-Banach algebra under convolution and the following properties hold:\\
(i) $\cC_{\cB}$ is a q{u}asi-Banach algebra under matrix multiplication (equivalently, under composition of the associated operators).\\
(ii) Let $\mathcal{Y}$ a solid quasi-Banach space of sequences on $\Lambda$. If $\cB\ast\mathcal{Y}\subseteq\mathcal{Y}$ then $\cC_{\cB}$ acts boundedly on $\mathcal{Y}$, that is 
\begin{equation}
	\|Ac\|_{\mathcal{Y}}\leq \|A\|_{\cC_{\cB}}\|c\|_{\mathcal{Y}},\quad\forall A\in\cC_{\cB},\,c\in\mathcal{Y}.
	\end{equation}
(iii) Since $\cB\subseteq\ell^1(\Lambda)$ we may identify $\cC_{\cB}$ as a quasi-Banach subalgebra of $\cB(\ell^2(\Lambda))$.
\end{lemma}
Observe that $\cB$ is commutative whereas $\cC_{\cB}$ is not, that is why the passage from $\cB$ to $\cC_{\cB}$ can be viewed as a non-commutative extension of convolution quasi-algebras of sequences on $\Lambda$.
Crucial question about $\cC_{\cB}$ is whether it is inverse closed. 
\begin{definition}
	Let $\cB\subseteq\cA$ two quasi-Banach algebras with common unit element.  Then $\cB$ is inverse-closed in $\cA$ if $b\in\cB$ and $b^{-1}\in\cA$ implies that $b\in\cB$.
\end{definition}

The following theorem gives a characterization of the inverse closedness of $\cC_{\cB}$. The proof for $\cB=\ell^1_{v_s}(\Lambda)$ is due to Baskakov \cite{B1990}, see the general algebra case in \cite[Theorem 3.5]{GR}.

We shall give a detailed proof of the result below for the quasi-algebras cases $0<q<1$. This result is valuable of its own and could find applications in other frameworks. The proof follows the same pattern as in \cite{GR}, but the tools involved needed to be extended to the  quasi-Banach algebras setting. We devote the appendix below to prove those results. By a basis change for the lattice $\Lambda$, we assume without loss of generality that $\Lambda=\zdd$.
Moreover, for the sake of generality, the following theorem is stated for the  dimension $d$, namely $\cB=\ell^q_{v_s}(\zd)$.

A tool we shall need to prove Theorem \ref{c-i} is the Fourier transform of matrices $A=(a_{k,j})_{j,k\in\bZ^d}$. Let
\[
D_A(n)_{k,j}=\begin{cases}
	a_{k,k-n} & \text{if $j=k-n$},\\
	0 & \text{otherwise}
\end{cases} \qquad (n,j,k\in\bZ^d)
\]
be the $n$-th diagonal of $A$ and 
\[
M_tc(k)=e^{2\pi ik t}c(k) \qquad t\in\bT^d, k\in\bZ^d,
\]
where $c=(c(k))_{k\in\bZ}$ is a sequence, be the modulation operator, which is unitary on $\ell^2(\bZ^d)$ for all $t\in\bT^d$ and satisfies $M_{t+k}=M_t$ for all $k\in\bZ^d$.

For a matrix $A=(a_{k,j})_{j,k\in\bZ^d}$, we set
\begin{equation}\label{defft}
	f(t)=M_tAM_{-t} \qquad t\in\bT^d.
\end{equation}
We need the following result, whose proof for the quasi-Banach algebra case goes exactly as that of  \cite[Lemma 8.5]{GR}.
\begin{lemma}\label{lemmaGR85}  Let $\cB$ be a commutative  quasi-Banach algebra. Under the notation above,\\
	(i) $f(t)_{k,j}=a_{k,j}e^{2\pi i(k-j) t}$ for $k,j\in\bZ^d$ and $t\in\bT^d$;\\
	(ii) the matrix-valued Fourier coefficients of $f(t)$ are given by
	\[
	\hat f(n)=\int_{[0,1]^d}f(t)e^{-2\pi i n t}dt=D_A(n),
	\]
	with the appropriate interpretation of the integral, and $\norm{D_A(n)}_{op}=d_A(n)$;\\
	(iii) let $\cB(\bT^d,B(\ell^2))$ be the space of matrix-valued Fourier expansions $g$ that are given by $g(t)=\sum_{n\in\bZ^d}B_ne^{2\pi i n t}$, with $B_n\in B(\ell^2)$ and $(\norm{B_n}_{op})_{n\in\bZ^d}\in\cB$. Then,
	\[
	A\in\cC_\cB \qquad \Longleftrightarrow \qquad f(t)\in \cB(\bT^d,B(\ell^2)).
	\]
\end{lemma}

\begin{theorem}\label{c-i}
Consider $\cB=\ell^q_{v_s}(\zd)$, $0<q<1$, $s\geq0$. Then the following are equivalent:\\
	(i) $\cB$ is inverse-closed in $B(\ell^2)$.\\
	(ii) $\cC_\cB$ is inverse-closed in $B(\ell^2)$.\\
	(iii) The spectrum $\widehat{\cB}\simeq \bT^d$.	
\end{theorem}
\begin{proof} We first prove that $(i)$ and $(iii)$ are equivalent, then we turn to the other implications.\\
$(iii)\Rightarrow (i)$. Assume that $\widehat\cB\simeq\bT^d$ and let $a\in\cB$ such that $C_a$ is invertible in $B(\ell^2)$. We have to prove that $a$ has an inverse in $\cB$. Since $\widehat\cB\simeq\bT^d$, the restriction of the Gelfand transform to $\bT^d$ is the Gelfand transform itself, so that $\cF a$ coincides with the Gelfand transform by Proposition \ref{e2} $(iii)$. 
Hence, by Proposition \ref{e2} $(ii)$, the Fourier series of $a$ does not vanish at any point, which means that the Gelfand transform does not vanish at any point. By Theorem \ref{e1}, it follows that $a$ is invertible in $\cB$.\\
	$(i) \Rightarrow (iii)$. Assume $\widehat{\cB}\not\simeq\bT^d$. Since $\cB\subset \ell^1$ then $\widehat{\ell^1}\simeq\bT^d\subset \widehat{\cB}$ and the Fourier series of any elements of $\cB$ is the restriction to the strict subset $\bT^d$ of its Gelfand transform, so they do not coincide unless the Gelfand transform vanishes on $\widehat{\cB}\setminus\bT^d$. Assume that $\cB$ is inverse closed in $B(\ell^2)$. 
	
	By Theorem \ref{thm115} $(iii)$, $a\in\cB$ is invertible if and only if the Gelfand transform of $a$ does not vanish on $\widehat{\cB}$. Moreover, by definition, $a$ is invertible in $\cB$ if and only if $C_a$ is invertible in $B(\ell^2)$ with inverse, say $C_b$, that satisfies $b\in \cB$. On the other hand, $C_a$ is invertible if and only if the Fourier series of $a$ does not vanish on $\bT^d$, by Proposition \ref{e2} $(ii)$. 
	
	But the Fourier series of $a$ is only the restriction of the Gelfand transform to $\bT^d$, so the invertibility of $a$ is not equivalent to that of $C_a$. This is a contradiction.\\
	$(ii)\Rightarrow (i)$. Since $C_ab(k)=a\ast b(k)=\sum_{j\in\bZ^d}a(k-j)b(j)$, $C_a$ has matrix $A$ with entries $a(k-j)$, $k,j\in\bZ^d$. Therefore, $$d_A(j)=\sup_k|A_{k,k-j}|=\sup_k|a(k-j+j)|=|a(j)|,$$ so that $\norm{A}_{\cC_{\cB}}=\norm{d_A}_{\cB}=\norm{a}_{\cB}$. Hence, $A\in\cC_\cB$ if and only if $a\in\cB$.
	
	Assume that $\cB$ is not inverse-closed in $B(\ell^2)$ and let $a\in\cB$ be such that $C_a$ is invertible on $\ell^2$ with inverse $C_b$, with $b\notin\cB$. By the previous argument, the matrix $B$ of $C_b$ cannot be in $\cC_\cB$, that means that $\cC_\cB$ cannot be inverse-closed in $B(\ell^2)$.\\
	$(iii)\Rightarrow(ii)$. Assume that $A\in\cC_\cB$ is invertible in $B(\ell^2)$, we have to prove that if $(iii)$ holds, the inverse of $A$ is in $\cC_\cB$. Let $f(t)=M_tAM_{-t}$ ($t\in\bT^d$) be defined as in (\ref{defft}). By Lemma \ref{lemmaGR85} $(iii)$, $f(t)$ has a $B(\ell^2)$- valued Fourier series
	\begin{equation}\label{eq87}
	f(t)=\sum_{n\in\bZ^d}D_A(n)e^{2\pi in t},
	\end{equation}
	where $D_A(n)$ the $n$-th diagonal of $A$, and $\norm{D_A(n)}_{op}=d_A(n)$ is in $\cB$. We identify $\cB$ with a sub-quasi-algebra of $\cB(\bT^d,B(\ell^2))$ via the embedding $\iota:\cB\to\cB(\bT^d,B(\ell^2))$ defined for all $a\in\cB$ and all $t\in\bT^d$ as
	\[
		\iota(a)(t)=\sum_{n\in\bZ^d}a(n)e^{2\pi in t}I=\hat a(t)I,
	\]
	where $\hat a$ is the Fourier transform of $a\in\cB$, which coincides with the Gelfand transform by the validity of $(iii)$ and $I$ is the identity operator.  Let $\cM$ be a maximal left ideal of $\cB(\bT^d,B(\ell^2))$ and $\pi_\cM$ be the corresponding representation. Since $\iota(a)$ is a multiple of $I$, $\iota(a)$ commutes with every element of $\cB(\bT^d,B(\ell^2))$, we find that for all $T\in\cB(\bT^d,B(\ell^2))$ and all $a\in\cB$,
	\[
		\pi_\cM(T)\pi_\cM(\iota(a))=\pi_\cM(\iota(a))\pi_\cM(T).
	\]
	By Lemma \ref{8.9}, $\pi_\cM(\iota(a))$ must be a multiple of the identity, and since $\pi_\cM$ is a homomorphism, there exists a multiplicative linear functional $\chi\in\widehat{\cB}$ such that $\pi_\cM(\iota(a))=\chi(a)I$. Since $\widehat{\cB}\simeq\bT^d$, and $\chi\in\widehat{\cB}$, there exists $t_0\in\bT^d$ such that $\chi(a)=\hat a(t_0)$. Consequently,
	\[
		\pi_{\cM}(\iota(a))=\hat a(t_0)I, \qquad  a\in\cB.
	\]
	Let $\delta_n$ be the standard basis of $\ell^1(\bZ^d)$. Since $\cB$ is solid, $\delta_n\in\cB$ and $\iota(\delta_n)(t)=e^{2\pi in t}I$. By (\ref{eq87}), $f=\sum_{n\in\bZ^d}D_A(n)\iota(\delta_n)$, so that
	\[
		\begin{split}
			\pi_\cM(f)&=\pi_\cM\left(\sum_{n\in\bZ^d}D_A(n)\iota(\delta_n)\right)=\sum_{n\in\bZ^d}\pi_\cM(D_A(n))\pi_{\cM}(\iota(\delta_n))\\
			&=\sum_{n\in\bZ^d}\pi_\cM(D_A(n))e^{2\pi in t_0}I=\pi_\cM\left(\sum_{n\in\bZ^d}D_A(n)e^{2\pi in t_0}\right)\\
			&=\pi_\cM(f(t_0)).
		\end{split}
	\]
	Since the modulations $M_t$ are unitary, if $A\in\cC_\cB$ is invertible in $B(\ell^2)$, so is $f(t)=M_tAM_{-t}$ for all $t\in\bT^d$. By Lemma \ref{8.8}, $\pi_\cM(f(t_0))$ is left-invertible for every maximal left ideal in $\cB(\bT^d,B(\ell^2))$. Equivalently, $\pi_\cM(f)$ is invertible for every maximal left ideal in $\cB(\bT^d,B(\ell^2))$. By Lemma \ref{8.8}, $f(t)$ is invertible in $\cB(\bT^d,B(\ell^2))$. By definition of $\cB(\bT^d,B(\ell^2))$, this means that $f(t)^{-1}$ possesses a Fourier series
	\[
		f(t)^{-1}=M_tA^{-1}M_{-t}=\sum_{n\in\bZ^d}B_ne^{2\pi in t}
	\]
	with $(\norm{B_n})_{n\in\bZ^d}\in\cB$. By Lemma \ref{lemmaGR85} $(ii)$, $B_n$ is the $n$-th side diagonal of $A^{-1}$. As a consequence, Lemma \ref{lemmaGR85} $(iii)$ implies that $A^{-1}\in\cC_\cB$. 
\end{proof}

Corollary 3.7 of \cite{GR} works also  for quasi-algebras, the proof uses Lemma \ref{3.4} and it is exactly  the same. 
\begin{corollary}[Spectral Invariance] 
	Consider the (quasi-)algebra $\cB$ above. Assume $\widehat{\cB}\simeq \bT^d$, then
	\begin{equation}\label{3.5}
	{\rm Sp}_{\cB(\ell^2)}(A)=	{\rm Sp}_{\cC_{\cB}}(A),\quad \forall\, A\in\cC_{\cB}.
	\end{equation}
If $\cB$ acts boundedly on a solid sequence space $Y$, then
\begin{equation}\label{3.6}
	{\rm Sp}_{\cB(Y)}(A)\subseteq  {\rm Sp}_{\cB(\ell^2)}(A),\quad \forall\, A \in\cC_{\cB} .
\end{equation}
\end{corollary}
\section{Almost diagonalization for Weyl operators}\label{sec:ADWO}
Fix a Parseval Gabor frame $\cG(g,\Lambda)$ with $g\in\cS(\rd)$, take $\sigma\in\cS'(\rdd)$ (or some suitable subspace) and let $M(\sigma)$ the matrix with entries 
\begin{equation}\label{Msigma}
	M(\sigma)_{\mu,\lambda}=\la Op_w(\sigma)\pi(\lambda)g,\pi(\mu)g\ra, \quad \lambda,\mu\in\Lambda.
\end{equation}
Following the notation in \cite{GR}, we  denote by 
$$V_g^{\Lambda}f(\lambda)=\la f ,\pi(\lambda)g\ra,$$
the restriction of the STFT of $f$ to the lattice $\Lambda$.  We can write
$$f=\sum_{\lambda\in\Lambda}\la f,\pi(\lambda)g\ra \pi(\lambda)g,$$
so that 
$$\la Op_w(\sigma) f,\pi(\mu) g\ra = \sum_{\lambda\in\Lambda} \la f, \pi(\lambda)g\ra \la Op_w(\sigma)\pi(\lambda)g,\pi(\mu)g\ra, $$
that is
\begin{equation}\label{discreta}
	V_g^\Lambda( Op_w(\sigma) f) =M(\sigma) V_g^{\Lambda} f.
\end{equation}
The commutation relation can be easily seen via the diagram
\begin{equation}\label{diagramma}
	\begin{diagram}
		\node{L^2(\rd)}\arrow{s,r}{V_g^{\Lambda}}
		\arrow{c,t}{ Op_w(\sigma)} \node{
			L^2(\rd)}\arrow{s,r}{V_g^{\Lambda}}
		\\
		\node{\ell^2(\Lambda)}
		\arrow{c,t}{M(\sigma)} \node{
			\ell^2(\Lambda)}
	\end{diagram}
\end{equation}
Our goal is to characterize the inverse of $Op_w(\sigma)$ in terms of the matrix operator $M(\sigma)$. Let us underline that that the invertibility of $Op_w(\sigma)$ on $\lrd$ does not guarantee the invertibility of $M(\sigma)$ on $\ell^2(\Lambda)$, see Lemma \ref{3.7} below. That is why we recall the definition of pseudo-inverse. 
\begin{definition}
	An operator $A:\ell^2\to \ell^2 $ is pseudo-invertible if there exists a closed invariant subspace $\mathcal{R} \subseteq \ell^2$, such that $A$ is invertible on {\rm ran}$A =\mathcal{R}$ and {\rm ker}$A=\mathcal{R}^\perp$. The unique operator $A^\dagger$ that satisfies $\mathcal{A}^\dagger Ah=AA^\dagger h=h$ for $h\in \mathcal{R}$ and {\rm ker} $A^\dagger=\mathcal{R}^\perp$ is called the pseudo-inverse of $A$.
\end{definition}

A consequence of Theorem \ref{c-i} is the property of  pseudo-inverses for elements in $\cC_{\cB}$. The proof is the same as in \cite[Lemma 5.4]{GR}.
\begin{lemma}\label{5.4}
	If $\cB$ is inverse-closed in $\cB(\ell^2)$ and $A\in\cC_{\cB}$ has a pseudo-inverse $A^\dagger$, then $A^\dagger\in\cC_{\cB}$.
\end{lemma}
We recall the following lemma \cite{charly06}:
\begin{lemma}\cite{charly06}\label{3.7}
	If $Op_w(\sigma)$  is bounded on $\lrd$  then $M(\sigma)$ is bounded on $\ell^2(\Lambda)$  
	and maps \textrm{\rm ran} $ V_g^{\Lambda}$ into \textrm{\rm ran} $V_g^{\Lambda}$ with \textrm{\rm ran} $(V_g^{\Lambda})^{\perp}\subseteq$ \textrm{\rm ker} $M(\sigma)$.   
	
	Let $T$ be a matrix such that $V_g^{\Lambda}(Op_w(\sigma)f)=TV_g^{\Lambda}f$ for all $f\in\lrd$. 
	
	If {\textrm{\rm ran}} $(V_g^{\Lambda})^{\perp} \subseteq $ \textit{\rm ker} $T$, then $T=M(\sigma)$.
\end{lemma}

In what follows we need the characterization for Weyl operators with symbols in $M^{\infty,q}_{1\otimes v_s}(\rdd)$ which is contained in \cite{BC2021}. A direct inspection of the proof allows to replace the dominating function $H\in L^q_{v_s}(\rdd)$ with one in the smoother space $W(\cC,\ell^q_{v_s})(\rdd)$. 
\begin{theorem}\label{teor41}
	Consider  $g\in \cS(\rd)\setminus\{0\}$ and a lattice $\Lambda \subset\mathbb{R}^{2d}$ such that  $\mathcal{G}\left(g,\Lambda\right)$ is a  Gabor frame for $L^{2}\left(\mathbb{R}^{d}\right)$. For any $s\in\bR$, $0<q\leq\infty$, the following properties are equivalent:
	\begin{enumerate}
		\item[$(i)$] $\sigma\in M_{1\otimes v_s}^{\infty,q}\left(\mathbb{R}^{2d}\right)$.
		\item[$(ii)$] $\sigma\in\mathcal{S}'\left(\mathbb{R}^{2d}\right)$ and there exists
		a function $H\in W(\cC,\ell^q_{v_s})(\rdd)$ such that
		\begin{equation}
			\left|\left\langle Op_w(\sigma) \pi\left(z\right)g,\pi\left(u\right)g\right\rangle \right|\le  H(u-z), \qquad\forall u,z\in\mathbb{R}^{2d}.\label{eq:almdiag J}
		\end{equation}
		\item[$(iii)$] $\sigma\in\mathcal{S}'\left(\mathbb{R}^{2d}\right)$ and there exists
		a sequence $h\in \ell^q_{v_{s}}(\Lambda)$  such that
		\begin{equation}
			\left|\left\langle Op_w(\sigma) \left(\sigma\right)\pi\left(\mu\right)g,\pi\left(\lambda\right)g\right\rangle \right|\le C h( \lambda-\mu), \qquad\forall\lambda,\mu\in\Lambda.\label{eq:almdiag discr}
		\end{equation}
	\end{enumerate}
\end{theorem}
\begin{theorem}\label{4.7}
For $0<q\leq 1$, we have $\sigma\in M_{1\otimes v_s}^{\infty,q}(\rdd)$ if and only if $M(\sigma)\in\cC_{\cB}$ with equivalence of norms:
\begin{equation}
	\|M(\sigma)\|_{\cC_{\cB}}\asymp \|\sigma\|_{M_{1\otimes v_s}^{\infty,q}}.
\end{equation}
	\end{theorem}
\begin{proof}
	It is a consequence of the equivalence $(i)\Leftrightarrow (iii)$ of Theorem \ref{teor41}. The algebra case $q=1$ is proved in \cite{GR}.
\end{proof}
\begin{theorem}\label{6.8}
	The class of Weyl operators with symbols in $M^{\infty, q}_{1\otimes v_s}(\rdd)$, $0< q\leq 1$, is inverse-closed in $B(L^2(\rd))$. In other words, if $\sigma\in M^{\infty, q}_{1\otimes v_s}(\rdd)$ and $Op_w(\sigma)$ is invertible on $\lrd$, then $(Op_w(\sigma))^{-1}=Op_w(b)$ for some $b\in M^{\infty, q}_{1\otimes v_s}(\rdd)$. 
\end{theorem}
\begin{proof}  Assume $Op_w(\sigma)$ is invertible on $\lrd$ for some $\sigma\in M^{\infty, q}_{1\otimes v_s}(\rdd)$. Let $\tau \in\cS'(\rdd)$ be the unique distribution such that $Op_w(\sigma)^{-1}=Op_w(\tau)$. We shall prove that  $\tau\in M^{\infty, q}_{1\otimes v_s}(\rdd)$. Since $Op_w(\tau)$ is bounded on $\lrd$, Lemma \ref{3.7} implies that the infinite matrix operator $M(\tau)$  is bounded on $\ell^2(\Lambda)$ and maps \textrm{\rm ran} $V_g^{\Lambda}$ into itself with {\textrm{\rm ran}} $(V_g^{\Lambda})^{\perp} \subseteq $ \textit{\rm ker} $T$. If $f\in\lrd$, then by \eqref{discreta} we can write
	$$M(\tau)M(\sigma)V_g^{\Lambda}f=M(\tau)V_g^{\Lambda}(Op_w(\sigma)f)=V_g^\Lambda (Op_w(\tau)Op_w(\sigma)f)=V_g^{\Lambda}f.$$
	Hence $M(\tau)M(\sigma)={\rm Id}$ on {\textrm{\rm ran}} $V_g^{\Lambda}$ and $M(\tau)M(\sigma)=0$ on {\textrm{\rm ran}} $(V_g^{\Lambda})^{\perp}$. Likewise, $M(\sigma)M(\tau)={\rm Id}$ on {\textrm{\rm ran}} $V_g^{\Lambda}$ and $M(\sigma)M(\tau)=0$ on {\textrm{\rm ran}} $(V_g^{\Lambda})^{\perp}$. Hence $M(\tau)=M(\sigma)^\dagger.$
	
	 By Theorem \ref{4.7}, if $\sigma\in M_{1\otimes v_s}^{\infty,q}(\rdd)$ then $M(\sigma)\in \cC_{\cB}$ and Lemma \ref{5.4} gives $M(\tau)=M(\sigma)^\dagger\in \cC_{\cB}$. By Theorem \ref{teor41} we conclude that $\tau\in M^{\infty, q}_{1\otimes v_s}(\rdd)$.
\end{proof}

\section{Generalized metaplectic operators}\label{sec:GMO}
The theory developed so far find application in the framework of generalized metaplectic operators. In what follows we shall show the invertibility property and the explicit representation of such operators. 
\begin{theorem}[Invertibility in the class $FIO(\chi,q,v_s)$]\label{inverse}
	Consider $T\in FIO(\chi,q,v_s),$ such that $T$ is invertible on $L^2(\rd)$,
	then $T^{-1} \in FIO(\chi^{-1},q,v_s)$.
\end{theorem}
\begin{proof} The pattern is similar to Theorem 3.7 in \cite{CGNRJMPA}. We detail the differences. We first show that the adjoint operator $T^\ast$ 
	belongs to the class $FIO(\chi \inv, q,v_s)$. By Definition \ref{def1.1}:
	\begin{align*}|\langle T^*\pi(z)g,\pi(w) g \rangle|&=|\langle \pi(z)g,T(\pi(w) g) \rangle|=|\langle T(\pi(w) g,\pi(z)g) \rangle|\\
	&\leq H( z-\chi(w))=\cI (H\circ\chi)(w-\chi^{-1}z).
	\end{align*}
Observe that $\cI (H\circ\chi)\in W(C,\ell^q_{v_s})$ for $H\in W(C,\ell^q_{v_s})$ by Lemma \ref{lemmaRicoprimento}, since $v_s\circ \chi^{-1}\asymp v_s$, and the claim follows.
	Hence, by Theorem \ref{listprop} $(ii)$, the operator $P:=T^\ast T$ is in
	$FIO(\mathrm{Id},q,{v_s})$ and satisfies the estimate \eqref{unobis2}, that is: $$|\langle P\pi
	(\lambda )g, \pi (\mu )g\rangle | \leq h (\lambda - \mu
	), \quad\forall \lambda , \mu \in \Lambda, $$
	and a suitable sequence $h\in \ell^q_{v_s}(\Lambda)$.
	The characterization for pseudodifferential operators  in
Theorem 3.2 \cite{BC2021} says that $P$ is a Weyl operator  $P=Op_w(\sigma)$ with a symbol $\sigma$ in
	$M^{\infty,q}_{1\otimes v_s}(\rdd)$. Since $T$ and therefore
	$T^\ast$ are invertible on $L^2(\rd)$, $P$ is also invertible on
	$L^2(\rd)$. Now we apply Theorem~\ref{6.8} and conclude that the inverse
	$P^{-1}=Op_w(\tau)$ is a  Weyl 
	operator with symbol in $\tau\in M^{\infty,q}_{1\otimes v_s}(\rdd)$.
	Hence $P^{-1}$ is in $FIO(\mathrm{Id},q,v_s)$. Eventually, using the
	algebra property of Theorem \ref{listprop} $(ii)$, we obtain that $T^{-1}=P^{-1}
	T^\ast$ is in $FIO(\chi \inv ,q,v_s)$.
\end{proof}

\begin{theorem}\label{pseudomu}
	Fix $0< q\leq 1$, $\chi \in \Spnr $. A linear continuous operator $T:
	\cS(\rd)\to\cS'(\rd)$ is in  $FIO(\chi,q,v_s)$ if and only if there
	exist symbols $\sigma_1, \sigma_2 \in
	M^{\infty,q}_{1\otimes v_s}(\rdd)$, 
	such that
	\begin{equation}\label{pseudomu1}
	T=Op_w(\sigma_1)\mu(\chi)\quad \mbox{and}\quad
	T=\mu(\chi)Op_w(\sigma_2).
	\end{equation}
	The symbols $\sigma _1$ and $\sigma _2$ are related by
	\begin{equation}\label{hormander}\sigma_2=\sigma_1\circ\chi.\end{equation}
\end{theorem}
\begin{proof}
	It follows the same pattern of the proof of \cite[Theorem 3.8]{CGNRJMP2014}. The main tool is the characterization in Theorem 3.2 of \cite{BC2021} which extends Theorem 4.6 in \cite{GR} to the case $0<q<1$. We recall the main steps for the benefit of the reader. \par 
	Assume $T\in FIO(\chi,q,v_s)$ and fix $g\in\cS(\rd)$. We first prove the factorization $T=\sigma_1^w\mu(\chi)$.
	For every $\chi\in\Spnr$, the kernel of $\mu (\chi )$ with respect to
	\tfs s can be written as
	$$|\langle \mu(\chi) \pi(z)g,\pi(w)g\rangle|=|V_{g}\big(\mu(\chi)g\big)\big(w-\chi z\big)|.$$
	Since both  $g\in \cS (\rd )$ and
	$\mu(\chi)g\in\cS(\rd)$, we have
	$V_{g}(\mu(\chi)g)\in\cS(\rdd)$ (see e.g., \cite{Elena-book}).
	Consequently, we have found a function $H=|V_{g}\big(\mu(\chi)g\big)|\in \cS(\rdd)\subset W(C, \ell^q_{v_s})$
	such that
	\begin{equation}
	|\langle \mu(\chi) \pi(z)g,\pi(w)g\rangle| \leq H(w-\chi z) \, \quad w,z\in \rdd . \label{stimam}
	\end{equation}
	Since $\mu(\chi)^{-1}=\mu(\chi^{-1})$ is in $FIO( \chi \inv ,q,v_s)$ by
	Theorem \ref{inverse}, the algebra property of Theorem~\ref{listprop} $(ii)$
	implies that
	$T\mu(\chi^{-1})\in FIO(\mathrm{Id},q,v_s)$. Now Theorem 3.2 in \cite{BC2021}
	implies the existence of a symbol
	$\sigma _1 \in M^{\infty,q}_{1\otimes v_s}(\rdd)$, such that
	$T\mu(\chi)^{-1}=Op_w(\sigma_1)$, as claimed. The rest goes exactly as in \cite[Theorem 3.8]{CGNRJMP2014}.
\end{proof}

\begin{appendix}
\section{Quasi-Banach algebras}
We consider here the solid involutive quasi-Banach algebras with respect to convolution $\cB=\ell^q_{v_{s}}(\Lambda)$, $s\geq 0$, $0<q\leq1$. For $q=1$ we recapture the algebra $\ell^1_{v_s}(\Lambda)$.  As before, without loss of generality, we may assume $\Lambda=\zdd$.

The unit element  is given by the sequence $\delta=(\delta(k))_{k\in\zdd}$, with elements
$$
\delta(k)=
	\begin{cases}
	1,\quad k=0\\
	0,\quad k\in\zdd\setminus\{0\},
	\end{cases}
$$
We have $\|\delta||_{\ell^q_{v_s}}=1$ for every $s\in\bR$. Moreover, for every  $a\in \ell^q_{v_s}(\zdd)$,  $$a\ast\delta(k)=\sum_{j\in\zdd} \delta(j)a(k-j) = \delta(0)a(k) =a(k),$$ $k\in\zdd$.

For sake of clarity, we first recall the general definition of a quasi-Banach space.
\begin{definition}\label{d1}
	Let be $X$ a complex vector space. A functional $\|\cdot\|: X\to [0,+\infty)$ is called {\bf quasinorm} if the following inequality holds
	\begin{equation}\label{QN}
	\|f+g\|\leq K (\|f\|+\|g\|),\quad \forall f,g\in X,
	\end{equation}
	where $K\geq1$, moreover, 
	$$\|f\|\geq 0\quad\mbox{and} \quad \|f\|=0\Leftrightarrow f=0$$
	and
	$$\|\lambda f\|=|\lambda|\|f\|,\quad \forall \lambda\in\bC,\,f\in X.$$
	The couple $(x,\|\cdot\|)$ is called a quasinormed space.  A complete quasinormed vector space is called a quasi-Banach space.
\end{definition}
Standard examples are $L^q$ spaces with $0<q<1$. In this case the functional $\|\cdot\|=\|\cdot\|_{L^q}$ is not a norm but satisfies \eqref{QN} with $K=2^{1/q}-1>1$ and it holds
\begin{equation}\label{QNq}
\|f+g\|^q\leq \|f\|^q+\|g\|^q,\quad \forall f,g\in L^q.
\end{equation}
A functional satisfying \eqref{QN} and \eqref{QNq} is called a $q$-norm. Relation \eqref{QNq} generalizes to 
\begin{equation}\label{QNqn}
	\|f_1+f_2+\cdots f_n\|^q\leq \sum\limits_{1}^{n} \|f_n\|^q,\quad \forall f_i\in L^q,\, i=1,\dots,n.
\end{equation}
If the metric $d(f,g)=|||f-g|||^q$ on $X$ defines a metric that induces the same topology on the quasi-Banach space $(X,\|\cdot\|)$, then  $X$ is also called $q$-Banach space. 
\begin{theorem}[Aoki–Rolewicz \cite{424}]
If $\|\cdot\|$ is a quasinorm on $X$, then there exist $q>0$ and a $q$-norm $|||\cdot|||$
on $X$ such that
$$ \frac1{C}\|f\|\leq |||f|||\leq \|f\|,\quad f\in X,$$
where $C>0$ is independent of $f$.
\end{theorem}
Following the pattern of \cite{Pavlovic}, from now on we assume that \emph{quasinorm} means $q$-norm, for some $q\in (0,1]$.

\subsection{General theory of quasi-Banach algebras}

\begin{definition}\label{QA}
	A \textbf{(complex) quasi-Banach algebra} $\cA$ is a complex vector space in which a multiplication $\cdot:\cA\times\cA\to\cA$ is satisfied so that \\
	(i) $x\cdot(y\cdot z)=(x\cdot y)\cdot z$,\\
	(ii) $(x+y)\cdot z=x\cdot z+y\cdot z$,\\
	(iii) $\alpha(x\cdot y)=(\alpha x)\cdot y=x\cdot(\alpha y)$
	for all $x,y,z\in\cA$ and $\alpha\in\bC$. In addition, $\cA$ is a quasi-Banach space with respect to a quasi-norm $\Vert\cdot\Vert$ that satisfies
	\begin{equation}\label{prod}
		\Vert x\cdot y\Vert\leq C_P\Vert x\Vert \Vert y\Vert
	\end{equation}
	for some $C_P>0$, and $\cA$ contains an element $e$ such that\\
	(iv) $x\cdot e=e\cdot x=x$;\\
	(v) $\Vert e\Vert=1$.
\end{definition}  

For $C_P=1$ \eqref{prod} becomes $\Vert x\cdot y\Vert\leq \Vert x\Vert \Vert y\Vert$ and we have the standard algebra property. In particular, if $C_P\leq1$ then the estimate $\Vert x\cdot y\Vert\leq \Vert x\Vert \Vert y\Vert$ holds as well. Thus, we limit to the case 
$$C_P\geq1.$$

In what follows, $\cA$ will always denote a quasi-Banach  agebra and $C_P$ will always denote the constant that appears in (\ref{prod}). Also, we denote with $C_S$ the constant in the definition of quasi-norm, namely
\[
\norm{x+y}\leq C_S\norm{x}\norm{y}.
\]

Examples for the case $C_S=1$ is given by $A=\ell^q_{v_s}(\zdd)$,  $0<q\leq 1$, $s\geq0$, which  satisfies:
$$\|x\ast y\|_{\ell^q_{v_s}}\leq \|x\|_{\ell^q_{v_s}} \|y\|_{\ell^q_{v_s}}.$$
Moreover $\ell^q_{v_s}(\zdd)$ are $q$-Banach spaces. 

From now on, we may assume without loss of generality that   $\Lambda=\zdd$.

\begin{remark}
	The multiplication $\cdot:\cA\times\cA\to\cA$ is continuous with respect to the quasi-norm topology on $\cA$ and left/right continuous. The proof goes exactly as in the Banach case. 
\end{remark}

\cite[Proposition 10.6]{RudinFunc} extends to the quasi-Banach case directly. For the following theorem in the Banach setting, we refer to \cite[Theorem 10.7]{RudinFunc}.

 Recall that a \textbf{complex homomorphism} on a quasi-Banach algebra $\cA$ is a linear mapping $\phi:\cA\to\bC$ such that $\phi\not\equiv0$ and $\phi(x\cdot y)=\phi(x)\phi(y)$ for all $x,y\in\cA$.

\begin{theorem}\label{thminv}
	Let $\cA$ be a  quasi-Banach algebra, $x\in\cA$, $\norm{x}<\frac{1}{C_P}$. Then,\\
	(i) $e-x$ is invertible in $\cA$ with inverse $s$;\\
	(ii) $\norm{s-e-x}\leq \frac{C_P^2\norm{x}^2}{(1-(C_P\norm{x})^q)^{1/q}}$;\\
	(iii) $|\phi(x)|<1$ for all complex homomorphism $\phi$ on $\cA$.
\end{theorem}
\begin{proof}
	(i) It follows precisely as in \cite[Theorem 10.7 (a)]{RudinFunc}, with 
	\[
		\norm{x^m+x^{m+1}+\ldots+x^n}^q\leq \sum_{j=m}^n\norm{x^j}^q\leq \sum_{j=m}^n(C_P\norm{x})^{qj},
	\]
	which goes to $0$ since the series converges. This proves that the partial sums $s_n=e+x+x^2+\ldots+x^n$ form a Cauchy sequence in $\cA$. Moreover, we also have $\norm{x^n}\to 0$ as $n\to+\infty$ because $$\norm{x^n}^q\leq C_P^{nq}\norm{x}^{nq}\to0 $$ since $C_P\norm{x}<1$ So, all the ingredients used to prove (i) are still valid.
	
	The proof of (ii) goes exactly as that of \cite[Theorem 10.7 (b)]{RudinFunc}, with the difference that
	\[
		\norm{s-e-x}^q=\norm{x^2+x^3+\ldots}^q\leq \sum_{j=2}^\infty (C_P)^{jq}\norm{x}^{jq}=\frac{C_P^{2q}\norm{x}^{2q}}{1-(C_P\norm{x})^q}.
	\]
	(iii) It is proved verbatim as in \cite[Theorem 10.7 (c)]{RudinFunc}.
\end{proof}

We denote with $G(\cA)$ the group of invertible elements of $\cA$. If $x\in\cA$, the \textbf{spectrum} of $x$ is defined exactly as in the Banach setting as
\[
\sigma(x)=\{\lambda\in\bC \ : \ \lambda e-x \ \mbox{is \ not \ invertible}\}.
\]
$\bC\setminus \sigma(x)$ is called the \textbf{resolvent} of $x$ and $\rho(x)=\sup_{\lambda\in\sigma(x)}|\lambda|$ is the \textbf{spectral radius} of $x$. The following result generalizes \cite[Theorem 10.11]{RudinFunc} to the quasi-Banach setting, and its proof is also a straightforward generalization.

\begin{theorem}
	Let $\cA$ be a quasi- Banach algebra, $x\in G(\cA)$ and $h\in\cA$ be such that $\norm{h}<\frac{1}{2C_P^2}\norm{x^{-1}}^{-1}$. Then, $x+h\in G(\cA)$ and
	\[
		\norm{(x+h)^{-1}-x^{-1}+x^{-1}hx^{-1}}\leq C_P^4\norm{x^{-1}}^3\norm{h}^2.
	\]
\end{theorem}
\begin{proof}
	Since $\norm{h}<\frac{1}{2C_P^2}\norm{x^{-1}}^{-1}$, 
	\[
		\norm{x^{-1}h}\leq C_P\norm{x^{-1}}\norm{h}<C_P\frac{1}{2C_P^2}\norm{x^{-1}}^{-1}\norm{x^{-1}}=\frac{1}{2C_P}<\frac{1}{C_P}
	\]
	By Theorem \ref{thminv}, $x^{-1}h$ is invertible in $\cA$ and
	\[\begin{split}
		\norm{(x+h)^{-1}-x^{-1}x^{-1}hx^{-1}}&=\norm{(e+x^{-1}h)^{-1}-e+x^{-1}h}\norm{x^{-1}}\\
		&\leq\frac{C_P^2\norm{x^{-1}h}^2}{(1-C_P^q\norm{x^{-1}h}^q)^{1/q}}\norm{x^{-1}}\\
		&\leq C_P^2\norm{x^{-1}h}^2\norm{x^{-1}}\leq C_P^4\norm{x^{-1}}^3\norm{h}^2.
	\end{split}
	\]
\end{proof}

As a consequence, $G(\cA)$ is open and $x\mapsto x^{-1}$ is a homomorphism of $G(\cA)$ onto itself, cf. \cite[Theorem 1.12]{RudinFunc}. It is also immediate to verify that \cite[Theorem 1.13]{RudinFunc} generalizes with the same statement, and the upper bound for $\rho(x)$ changes to 
\[
\rho(x)\leq C_P\norm{x} \qquad x\in\cA,
\]
in particular the \textit{spectral radius formula} holds: $$ \rho(x)=\lim_{n\to+\infty}\norm{x^n}^{1/n}=\inf_{n\geq1}\norm{x^n}^{1/n} $$ and the proof of the \cite[Theorem 10.14]{RudinFunc} extends to the quasi-Banach setting.

\begin{theorem}[Gelfand-Mazur]\label{thmGM}
	If $\cA$ is a quasi-Banach algebra and $G(\cA)=\cA\setminus\{0\}$, then $\cA$ is (isometrically) isomorphic to $\bC$. 
\end{theorem}
\begin{remark}
	The condition $\norm{e}=1$ serves in the proof of Theorem \ref{thmGM} to prove that the isomorphism $\lambda:\cA\to\bC$ of the Theorem of Gelfand-Mazur is an isometry. If $\norm{e}>0$, then $\lambda$ is a quasi-isometry, as $|\lambda(x)|=\norm{e}\norm{x}$ for all $x\in \cA$. Condition \ref{QA} (v) is barely used in this part of Banach quasi-algebras and, exactly as condition (\ref{prod}) with $C_P>0$, it has a minor impact on the validity of the Banach setting results. 
\end{remark}

\begin{definition}
	Let $\cA$ be a commutative complex quasi-Banach algebra. A linear subspace $\cJ\subseteq\cA$ is an \textbf{ideal} of $\cA$ if $x\cdot y\in\cJ$ for all $x\in\cA$ and all $y\in\cJ$. $\cJ$ is \textbf{proper} if $\cJ\neq\cA$ and it is \textbf{maximal} if it proper and it is not contained in any larger proper ideal.
\end{definition}
\cite[Proposition 11.2]{RudinFunc} and \cite[Theorem 11.3]{RudinFunc} extend trivially to the quasi-Banach setting. 

Let $\cJ$ be a closed and proper ideal of $\cA$. Let $\pi:\cA\to\cA/\cJ$ be the quotient map $\pi(a)=a+\cJ$ ($a\in\cA$). Define $$\norm{a+\cJ}:=\inf_{y\in J}\norm{a+y}.$$
 Then, $\norm{\cdot}$ defines a complex quasi-Banach algebra structure on $\cA/\cJ$. In fact, the product on $\cA/\cJ$ is defined precisely as in the Banach setting. Moreover, $\norm{\pi(x)}\leq\norm{x}$ since $0\in\cJ$, so $\pi$ is continuous with respect to the quasi-norm topologies. A slightly modification of the proof for the Banach setting leads to the inequality
\[
\norm{\pi(x)\pi(y)}\leq C_P\norm{\pi(x)}\norm{\pi(y)} \qquad \forall \pi(x),\pi(y)\in\cA/\cJ.
\]
Finally, $\norm{\pi(e)}=\norm{\pi(e)\pi(e)}\leq C_P\norm{\pi(e)}^2$, which implies that $\norm{\pi(e)}\geq 1/C_P$. If $C_P=1$, this implies that $\norm{\pi(e)}\geq 1$ and the other inequality follows trivially by the continuity of $\pi$.  If $C_P>0$, then we have
\[
\frac{1}{C_P}\leq\norm{\pi(e)}\leq 1.
\]
For this reason, when dealing with quotients quasi-algebras, condition (v) of Definition \ref{QA} can be replaced by  $\frac{1}{C_P}\leq\norm{{\pi}(e)}\leq 1$.

For our purposes $C_P=1$ and so also $\norm{\pi(e)}=1$.
\begin{theorem}\label{thm115}
	Let $\cA$ be a commutative quasi-Banach algebra and  $$\widehat{\cA}:=\{\phi:\cA\to\bC, \ complex \ homomorphism\}.$$ Then,\\
	(i) every maximal ideal of $\cA$ is the kernel of some $h\in\widehat{\cA}$,\\
	(ii) if $h\in\Delta$, $\ker(h)$ is a maximal ideal of $\cA$,\\
	(iii) $x\in \cA$ is invertible if and only if $h(x)\neq0$ for all $h\in\widehat{\cA}$,\\
	(iv) $x\in \cA$ is invertible if and only if $x$ lies in no proper ideal of $\cA$,\\
	(v) $\lambda\in\sigma(x)$ if and only if $h(x)=\lambda$ for some $h\in\widehat{\cA}$.
\end{theorem}
\begin{proof}
	Is just a readjustment of the proof of \cite[Theorem 11.5]{RudinFunc}.
\end{proof}

\begin{definition}
	Let $\cA$ be a commutative quasi-Banach algebra and $\widehat{\cA}$ be the set of the complex homomorphism of $\cA$. The \textbf{Gelfand transform} of $x\in \cA$ is the mapping $\hat x:\Delta\to\bC$ defined for all $h\in\widehat{\cA}$ as
	\[
	\hat x(h)=h(x).
	\]
\end{definition}

\begin{corollary}\label{corGT}
	Let $\cA$ be a commutative quasi-Banach algebra. Then, $x\in\cA$ is invertible if and only if $\hat x(h)\neq0$ for all $h\in\widehat{\cA}$. 
\end{corollary}
\begin{proof}
	It follows directly by Theorem \ref{thm115} (iii).
\end{proof}

Following the pattern of  Section 24 in \cite{Bonsall}, one can infer that the representation theory for quasi-Banach algebras goes exactly the same as for Banach algebras, since the main ingredients are the algebraic properties, the closedness criteria and the continuity of the representations. We then restate the same Lemmata 8.7, 8.8 and 8.9 in \cite{GR} in our setting as follows.

Let $\cA$ be a  quasi-Banach algebra with identity and $\cM\subseteq\cA$ a closed left ideal. Then $\cA$ acts on the quasi-Banach space $\cA/\cM$ by the left regular representation
\begin{equation}
	\pi_{\cM}(a)\tilde{x}=\widetilde{ax},\quad a\in\cA, \,\tilde{x}\in \cA/\cM,
\end{equation}
where $\tilde{x}$ is the equivalence class of $x$ in $\cA/\cM$. 
\begin{lemma}\label{8.7}
	If $\cM$ is a maximal left ideal of a quasi-Banach algebra  $\cA$, then $\pi_{\cM}$ is algebraically irreducible. That is, $$\{\pi_\cM(a)\tilde{x}\,: \,a\in\cA\}=\cA/\cM,$$
	for every $\tilde{x}\not=0$. 
\end{lemma}

\begin{lemma}\label{8.8}
	Let $\cA$ be a quasi-Banach algebra with identity. 
	An element $\cA$ is left-invertible (right-invertible) if and only if $\pi_{\cM}(a)$ is invertible for every maximal left (right) ideal $\cM\subseteq\cA$.
\end{lemma}

\begin{lemma}[Schur's Lemma for quasi-Banach space representations]\label{8.9}
	Assume that $\pi:\cA\to \cB(X)$ is an algebraically irreducible representation of $\cA$ on a quasi-Banach space $X$. If $T\in\cB(X)$ and $T\pi(a)=\pi(a)T$ for all $a\in\cA$, then $T$ is a multiple of the identity operator ${\rm Id}$ on $X$. 
\end{lemma}
\subsection{The quasi-Banch algebras $\cB$.}\par

Observe that, for $0<q\leq1$,
\begin{equation}\label{8.1}
	\ell^q_{v_s}(\zdd)\hookrightarrow \ell^2(\zdd),\quad s\geq0.
\end{equation}
(continuous embedding).

Let $\cD:=\{a\in\ell^2(\zd): \, \cF a\in L^\infty(\mathbb{T}^d)\}$ be the Banach algebra with the norm
$\|a\|_{\cD}:=\|\cF a\|_\infty,$ where 
\begin{equation}\label{Fs}
	\cF a(\xi)=\sum_{n\in\zd} a(n)e^{2\pi i n\xi}.
\end{equation}
\begin{lemma}\label{8.2}
$\cB$ is continuously embedded in $\cD$ and $\|a\|_{\cD}\leq \|a\|_{\cB}$.
\end{lemma}
\begin{proof}
	Since $\ell^q_{v_s}(\zd)\hookrightarrow \ell^1(\zd)$ with $\|a\|_{\ell^1}\leq \|a\|_{\ell^q_{v_s}}$ and $\ell^1(\zd)\hookrightarrow \cD$ with $\|a\|_{\cD}\leq \|a\|_{\ell^1}$, the result immediately follows.
\end{proof}

We recall a list of Lemmata from \cite{GR}. Namely,
\begin{lemma}[Lemma 8.3 \cite{GR}]\label{8.3}
	If $b\in\cD$ and $|a|\leq b$ then $a\in\cD$ and $\|a\|_{\cD}\leq\|b\|_{\cD}$. 
\end{lemma}
\begin{lemma}[Lemma 8.4 \cite{GR}]\label{8.4}
	Let $a$ be a sequence on $\zd$ such that $\cF|a|$ is well defined. Then
	\begin{equation}\label{8.1}
		\|a\|_1=\|\cF|a|\|_\infty.
	\end{equation}
\end{lemma}

\begin{proposition}\label{e2}
(i) The Gelfand transform of $a\in\ell^1(\zd)$ coincides with the Fourier series $\cF a$ in \eqref{Fs}.\\
(ii) The convolution operator $C_a b=a\ast b$ for $a\in\ell^1(\zd)$ is invertible  if and only if the Fourier series \eqref{Fs} does not vanish at any $\xi\in\mathbb{T}^d$.\\
(iii) If $a\in\cB$, then the restriction of the Gelfand transform of $a$ to $\bT^d$ is the Fourier series $\cF a$ of $a$. 
\end{proposition}
\begin{proof} For Items (i) and (ii) see \cite{GR}.  Item (iii) follows from the inclusion $\ell^q_{v_s}(\zd)\subseteq \ell^1(\zd)$, $0<q\leq1$. We know  that if $\cB\subseteq\ell^1$, then $\bT^d\subseteq\widehat{\cB}$ and the Fourier transform is the restriction of the Gelfand transform on $\bT^d$, so (iii) hold if $\bT^d\simeq\widehat{\cB}$.
\end{proof}
 
As a consequence of Corollary \ref{corGT} and Proposition \ref{e2}, we have 
 \begin{theorem}\label{e1}
 	Assume that $\widehat{B}\simeq\bT^d$. An element $a\in\cB$ is invertible if and only if its Fourier series $\cF a$ does not vanish at any point.
 \end{theorem}

\end{appendix}

\end{document}